\documentclass[11pt,leqno]{amsart}

\usepackage{amssymb,amsmath}

\newtheorem{theorem}{Theorem}[section]
\newtheorem{lemma}[theorem]{Lemma}
\newtheorem{corollary}[theorem]{Corollary}
\newtheorem{proposition}[theorem]{Proposition}

\theoremstyle{definition}
\newtheorem{definition}[theorem]{Definition}
\newtheorem{example}[theorem]{Example}

\theoremstyle{remark}
\newtheorem{remark}[theorem]{Remark}

\numberwithin{equation}{section}

\newcommand{\complex}{\mathbb{C}}
\newcommand{\C}{\mathbb{C}}

\newcommand{\Z}{\mathbb{Z}}
\newcommand{\R}{\mathbb{R}}

\newcommand{\aalg}{\mathcal{A}}
\newcommand{\rring}{\mathcal{R}}

\newcommand{\gcrystalline}{\aalg_0 \Diamond_{\sigma}^{\alpha} G}
\newcommand{\gskewring}{\mathcal{R}_e \rtimes_{\sigma} G}

\newcommand{\cstarcross}{C(X) \stackrel{C^*}{\rtimes_{\tilde{h}}} \Z}

\DeclareMathOperator{\identity}{id}

\DeclareMathOperator{\supp}{supp}

\DeclareMathOperator{\Per}{Per}
\DeclareMathOperator{\Aper}{Aper}

\DeclareMathOperator{\Pic}{Pic}
\DeclareMathOperator{\Aut}{Aut}

\begin{document}

%\title[Maximal commutativity in simple generalized crossed products]{A note on maximal commutativity in simple generalized crossed products}
%\title[Maximal commutativity in simple group graded rings]{Maximal commutativity in simple group graded rings}
\title[Simple group graded rings and maximal commutativity]{Simple group graded rings and maximal commutativity}

\author{Johan \"Oinert}
\address{Centre for Mathematical Sciences, Lund University, P.O. Box 118, SE-22100 Lund, Sweden}
%\curraddr{Department of Mathematics and Statistics,
%Case Western Reserve University, Cleveland, Ohio 43403}
\email{Johan.Oinert@math.lth.se}
\thanks{This work was partially supported by The Swedish Research Council, The Crafoord Foundation, The Royal Physiographic Society in Lund, The Swedish Royal Academy of Sciences, The Swedish Foundation of International Cooperation in Research and Higher Education (STINT) and "LieGrits", a Marie Curie Research Training Network funded by the European Community as project MRTN-CT 2003-505078. The author wishes to thank Magnus Goffeng, Patrik Lundstr\"om, Sergei Silvestrov and in particular Christian Svensson for useful discussions on the topic of this paper.}

%%    Information for second author
%\author{Author Two}
%\address{Mathematical Research Section, School of Mathematical Sciences,
%Australian National University, Canberra ACT 2601, Australia}

%    General info
%\subjclass[2000]{Primary 13A02, 16S35, 16W99; Secondary 17C20, 14C22}
%\subjclass{Primary 54C40, 14E20; Secondary 46E25, 20C20}
%\date{January 14, 2009 and, in revised form, ?, 2009.}
%\date{January 1, 1994 and, in revised form, June 22, 1994.}



\begin{abstract}
In this paper we provide necessary and sufficient conditions for strongly group graded rings to be simple. For a strongly group graded ring $\rring = \bigoplus_{g\in G} \rring_g$ the grading group $G$ acts, in a natural way, as automorphisms of the commutant of the neutral component subring $\rring_e$ in $\rring$ and of the center of $\rring_e$. We show that if $\rring$ is a strongly $G$-graded ring where $\rring_e$ is maximal commutative in $\rring$, then $\rring$ is a simple ring if and only if $\rring_e$ is $G$-simple (i.e. there are no nontrivial $G$-invariant ideals). We also show that if $\rring_e$ is commutative (not necessarily maximal commutative) and the commutant of $\rring_e$ is $G$-simple, then $\rring$ is a simple ring. These results apply to $G$-crossed products in particular. A skew group ring $\mathcal{R}_e \rtimes_{\sigma} G$, where $\mathcal{R}_e$ is commutative, is shown to be a simple ring if and only if $\mathcal{R}_e$ is $G$-simple and maximal commutative in $\mathcal{R}_e \rtimes_{\sigma} G$. As an interesting example we consider the skew group algebra $C(X) \rtimes_{\tilde{h}} \Z$ associated to a topological dynamical system $(X,h)$. We obtain necessary and sufficient conditions for simplicity of $C(X) \rtimes_{\tilde{h}} \Z$ with respect to the dynamics of the dynamical system $(X,h)$, but also with respect to algebraic properties of $C(X) \rtimes_{\tilde{h}} \Z$. Furthermore, we show that for any strongly $G$-graded ring $\rring$ each nonzero ideal of $\rring$ has a nonzero intersection with the commutant of the center of the neutral component.
%This paper is a sample prepared to illustrate the use of the American
%Mathematical Society's \LaTeX{} document class \texttt{amsproc} and
%publication-specific variants of that class for AMS-\LaTeX{} version 1.2.
\end{abstract}

\maketitle

% Bara tillfälligt för att få bättre överblick
%\setcounter{tocdepth}{4}
%\tableofcontents

\section{Introduction}

The aim of this paper is to highlight the important role that maximal commutativity of the neutral component subring plays in a strongly group graded ring when investigating simplicity of the ring itself. The motivation comes from the theory of $C^*$-crossed product algebras associated to topological dynamical systems. To each topological dynamical system, $(X,h)$, consisting of a compact Hausdorff space $X$ and a homeomorphism $h : X \to X$, one may associate a $C^*$-crossed product algebra\footnote{To avoid confusion, we let $\cstarcross$ denote the $C^*$-crossed product algebra in contrast to the (algebraic) skew group algebra, which is denoted $C(X) \rtimes_{\tilde{h}} \Z$.} $\cstarcross$ (see e.g. \cite{TomiyamaBook}). In the recent paper \cite{ChristianTomiyama}, C. Svensson and J. Tomiyama proved the following theorem.

\begin{theorem}\label{ChrissyTommySats}
The following assertions are equivalent:
\begin{enumerate}
	\item[(i)] $(X,h)$ is topologically free (i.e. the aperiodic points are dense in $X$).
	\item[(ii)] $I \cap C(X) \neq \{0\}$ for each nonzero ideal $I$ of $\cstarcross$.
	\item[(iii)] $C(X)$ is a maximal commutative $C^*$-subalgebra of $\cstarcross$.
\end{enumerate}
\end{theorem}

This theorem is a generalization (from closed ideals to arbitrary ideals) of a well-known theorem in the theory of $C^*$-crossed products associated to topological dynamical system (see e.g. \cite{TomiyamaBook} for details). Theorem \ref{ChrissyTommySats} is very useful when proving the following theorem, which originally appeared in \cite{Power}.

\begin{theorem}\label{simplicityCstar}
Suppose that $X$ is infinite. $\cstarcross$ is simple if and only if $(X,h)$ is minimal (i.e. each orbit is dense in $X$).
\end{theorem}

In the theory of graded rings, one theorem which provides sufficient conditions for a strongly group graded ring to be simple, is the following which was proven by F. Van Oystaeyen in \cite[Theorem 3.4]{VanOystaeyen}.

\begin{theorem}\label{stronglysimple}
Let $\mathcal{R} = \bigoplus_{g\in G} \mathcal{R}_g$ be a strongly $G$-graded ring such that the morphism $G \to \Pic(\mathcal{R}_e)$, defined by $g \to [\mathcal{R}_g ]$, is injective. If $\mathcal{R}_e$ is a simple ring, then $\mathcal{R}$ is a simple ring.
\end{theorem}

In \cite{OinSil,OinSilAGMFGBG,OinSil3,OinSil4} an extensive investigation of the intersection between arbitrary nonzero ideals in various types of graded rings and certain subrings, has been carried out. Given a subset $S$ of a ring $\rring$ we denote by $C_\rring(S)$ the commutant of $S$ in $\rring$, i.e. the set of all elements of $\rring$ which commute with each element in $S$. In the recent paper \cite{OinSilTheVav}, the following theorem was proven.

\begin{theorem}\label{stronglycommutantsnitt}
If $\mathcal{R} = \bigoplus_{g\in G} \mathcal{R}_g$ is a strongly $G$-graded ring, where $\mathcal{R}_e$ is commutative, then
\begin{displaymath}
	I\cap C_{\rring}(\mathcal{R}_e) \neq\{0\}
\end{displaymath}
for each nonzero ideal $I$ in $\mathcal{R}$.
\end{theorem}

This implies that if $\rring$ is a strongly $G$-graded ring where $\rring_e$ is maximal commutative in $\rring$, then each nonzero ideal in $\rring$ has a nontrivial intersection with $\rring_e$. 
For skew group rings the following was shown in \cite[Theorem 3]{OinSil4}.

\begin{theorem}\label{skewrings}
Let $\mathcal{R}=\gskewring$ be a skew group ring satisfying either of the following two conditions:
\begin{itemize}
\item $\mathcal{R}_e$ is an integral domain and $G$ is an abelian group.
\item $\mathcal{R}_e$ is commutative and $G$ is a torsion-free abelian group.
\end{itemize}
The following two assertions are equivalent:
\begin{enumerate}
	\item[(i)] The ring $\mathcal{R}_e$ is a maximal commutative subring in $\mathcal{R}$.
	\item[(ii)] $I \cap \mathcal{R}_e \neq \{0\}$ for each nonzero ideal $I$ in $\mathcal{R}$.
\end{enumerate}
\end{theorem}

This theorem can be seen as a generalization of the algebraic analogue of Theorem \ref{ChrissyTommySats} and it is applicable to the skew group algebra which sits densely inside the $C^*$-crossed product algebra $\cstarcross$.

The starting point of this paper was to consider the following proposition.

\begin{proposition}\label{maintheorem}
Let $\mathcal{R} = \bigoplus_{g\in G} \mathcal{R}_g$ be a strongly $G$-graded ring where $\mathcal{R}_e$ is maximal commutative in $\mathcal{R}$. If $\mathcal{R}_e$ is a simple ring, then $\mathcal{R}$ is a simple ring.
\end{proposition}

In Section \ref{ProofSection} we give three different proofs of Proposition \ref{maintheorem} and we shall in fact see that it is a very special case of Theorem \ref{stronglysimple}. The proofs are based on facts obtained in the preceding sections, and along the way we obtain new results on strongly graded rings in particular. In Section \ref{Preliminaries} we give definitions and background information necessary for the understanding of the rest of this paper. In Section \ref{intersectionnewresult} we generalize \cite[Corollary 3]{OinSilTheVav} and show that in a strongly $G$-graded ring $\rring$ each nonzero ideal has a nonzero intersection with $C_\rring(Z(\rring_e))$ (Theorem \ref{newzentersnitt}). Furthermore, we generalize \cite[Theorem 3]{OinSil4} and show that for a skew group ring $\gskewring$ where $\mathcal{R}_e$ is commutative, each nonzero ideal of $\gskewring$ has a nonzero intersection with $\mathcal{R}_e$ if and only if $\mathcal{R}_e$ is maximal commutative in $\gskewring$ (Theorem \ref{skewringsnew}).

The main objective of Section \ref{Simplicity} is to describe the connection between maximal commutativity of $\rring_e$ in a strongly group graded ring $\rring$ and injectivity of the canonical map $G \to \Pic(\rring_e)$. In Section \ref{GSimpleInCrystallineGradedRings} we show that if $\gcrystalline$ is a simple crystalline graded ring where $\aalg_0$ is commutative, then $\aalg_0$ is $G$-simple (Proposition \ref{CrystallineSimpleImplyGSimple}).
%We apply this theorem to the first Weyl algebra in Example \ref{FirstWeylAlgebraExample}.
In Example \ref{FirstWeylAlgebraExample} we apply this result to the first Weyl algebra.
In Section \ref{GSimpleSubringsStronglyGraded} we investigate simplicity of a strongly $G$-graded ring $\rring$ with respect to $G$-simplicity and maximal commutativity of $\rring_e$. In particular we show that if $\rring$ is a strongly $G$-graded ring where $\rring_e$ is maximal commutative in $\rring$, then $\rring_e$ is $G$-simple if and only if $\rring$ is simple (Theorem \ref{stronglyGgradedSimpleNecSuff}). We also show the slightly more general result in one direction, namely that that if $C_{\rring}(\rring_e)$ is $G$-simple (with respect to the usual action) and $\rring_e$ is commutative (not necessarily maximal commutative!), then $\rring$ is simple (Proposition \ref{GSimpleCommutantSimple}). In Section \ref{SimplicityOfSkewGroupRings} we investigate the simplicity of skew group rings and generalize \cite[Corollary 2.1]{Crow} and \cite[Theorem 2.2]{Crow}, by showing that if $\mathcal{R}_e$ is commutative, then the skew group ring $\gskewring$ is a simple ring if and only if $\mathcal{R}_e$ is $G$-simple and a maximal commutative subring of $\gskewring$ (Theorem \ref{skewringssimple}). As an example, we consider the skew group algebra associated to a dynamical system.

 In Section \ref{topdynsys} we consider the algebraic crossed product $C(X) \rtimes_{\tilde{h}} \Z$ associated to a topological dynamical system $(X,h)$. Under the assumption that $X$ is infinite, we show that $C(X) \rtimes_{\tilde{h}} \Z$ is simple if and only if $(X,h)$ is a minimal dynamical system or equivalently if and only if $C(X)$ is $\Z$-simple and maximal commutative in $C(X) \rtimes_{\tilde{h}} \Z$ (Theorem \ref{algebraicanalogue}). This result is a complete analogue to the well-known result for $C^*$-crossed product algebras associated to topological dynamical systems.

%Pedersen, Olesen
%
%Corollary 3.3
%Corollary 3.11

\section{Preliminaries}\label{Preliminaries}

Throughout this paper all rings are assumed to be unital and associative and unless otherwise is stated we let $G$ be an arbitrary group with neutral element $e$.

A ring $\mathcal{R}$ is said to be \emph{$G$-graded} if there is a family $\{\mathcal{R}_g\}_{g\in G}$ of additive subgroups of $\mathcal{R}$ such that
\begin{displaymath}
	\mathcal{R} = \bigoplus_{g\in G} \mathcal{R}_g  \quad \text{and} \quad \mathcal{R}_g \mathcal{R}_h \subseteq \mathcal{R}_{gh}
\end{displaymath}
for all $g,h\in G$. Moreover, if $\mathcal{R}_g \mathcal{R}_h = \mathcal{R}_{gh}$ holds for all $g,h \in G$, then $\mathcal{R}$ is said to be \emph{strongly $G$-graded}. The product $\rring_g \rring_h$ is here the usual module product consisting of all finite sums of ring products $r_g r_h$ of elements $r_g\in \rring_g$ and $r_h \in \rring_h$, and not just the set of all such ring products. For any graded ring $\rring$ it follows directly from the gradation that $\rring_e$ is a subring of $\rring$, and that $\rring_g$ is an $\rring_e$-bimodule for each $g\in G$. We shall refer to $\mathcal{R}_g$ as the \emph{homogeneous component of degree} $g\in G$, and in particular to $\rring_e$ as the \emph{neutral component}. Let $U(\rring)$ denote the group of multiplication invertible elements of $\rring$. We shall say that $\rring$ is a $G$-crossed product if $U(\rring) \cap \rring_g \neq \emptyset$ for each $g\in G$.

\subsection{Strongly $G$-graded rings}\label{MiyashitaAction}

For each $G$-graded ring $\mathcal{R}=\bigoplus_{g \in G}\mathcal{R}_{g}$ one has $1_{\mathcal{R}} \in \mathcal{R}_e$ (see \cite[Proposition 1.1.1]{MoGR}), and if we in addition assume that $\mathcal{R}$ is a strongly $G$-graded ring, i.e. $\mathcal{R}_{g} \mathcal{R}_{g^{-1}} = \mathcal{R}_{e}$ for each $g \in G$, then for each $g\in G$ there exists a positive integer $n_{g}$ and elements $a_{g}^{(i)} \in \mathcal{R}_{g}$, $b_{g^{-1}}^{(i)} \in
\mathcal{R}_{g^{-1}}$ for $i \in \{1,\ldots,n_g\}$, such that
\begin{equation}\label{partitionofunity}
\sum_{i=1}^{n_{g}}a^{(i)}_{g} \, b^{(i)}_{g^{-1}}=1_{\mathcal{R}}.
\end{equation}
For every $\lambda\in C_{\mathcal{R}}(\mathcal{R}_e)$, and in particular for every $\lambda\in Z(\mathcal{R}_e) \subseteq C_{\mathcal{R}}(\mathcal{R}_e)$, and
$g\in G$ we define

\begin{equation}\label{defofsigma}
	\sigma_g(\lambda)=\sum_{i=1}^{n_{g}}a^{(i)}_{g} \, \lambda \, b^{(i)}_{g^{-1}}.
\end{equation}
The definition of $\sigma_g$ is independent of the choice of the $a_g^{(i)}$'s and $b_{b^{-1}}^{(i)}$'s (see e.g. \cite{OinSilTheVav}).
For a proof of the following lemma we refer to \cite[Lemma 3]{OinSilTheVav}.

\begin{lemma}\label{karpilovskysats}
Let $\mathcal{R} = \bigoplus_{g\in G} \mathcal{R}_g$ be a strongly $G$-graded ring, $g\in G$ and write\\
$\sum_{i=1}^{n_{g}}a^{(i)}_{g}b^{(i)}_{g^{-1}}=1_{\mathcal{R}}$ for some $n_g > 0$ and $a_{g}^{(i)} \in \mathcal{R}_{g}$, $b_{g^{-1}}^{(i)} \in
\mathcal{R}_{g^{-1}}$ for $i \in \{1,\ldots,n_g\}$. For each $\lambda \in C_{\mathcal{R}}(\mathcal{R}_e)$ define $\sigma_g(\lambda)$ by $\sigma_g(\lambda)=\sum_{i=1}^{n_{g}}a^{(i)}_{g} \, \lambda \, b^{(i)}_{g^{-1}}$. The following properties hold:
\begin{enumerate}
\item[(i)] $\sigma_g(\lambda)$ is a unique element of $\mathcal{R}$ satisfying
\begin{equation}\label{sigmarelation}
	r_g \, \lambda = \sigma_g(\lambda) \, r_g, \quad \forall \, r_g \in \mathcal{R}_g.
\end{equation}
Furthermore, $\sigma_g(\lambda) \in C_{\mathcal{R}}(\mathcal{R}_e)$ and if $\lambda \in Z(\mathcal{R}_e)$, then $\sigma_g(\lambda) \in Z(\mathcal{R}_e)$.
\item[(ii)] The group $G$ acts as automorphisms of the rings $C_{\mathcal{R}}(\mathcal{R}_e)$ and $Z(\mathcal{R}_e)$, with each $g\in G$ sending any $\lambda \in C_{\mathcal{R}}(\mathcal{R}_e)$ and $\lambda \in Z(\mathcal{R}_e)$, respectively, into $\sigma_g(\lambda)$.
\item[(iii)] $Z(\mathcal{R})=\{\lambda \in C_{\mathcal{R}}(\mathcal{R}_e) \, \mid \, \sigma_g(\lambda)=\lambda, \,\, \forall g\in G\}$, i.e. $Z(\mathcal{R})$ is the fixed subring $C_{\mathcal{R}}(\mathcal{R}_e)^G$ of $C_{\mathcal{R}}(\mathcal{R}_e)$ with respect to the action of $G$.
\end{enumerate}
\end{lemma}

\noindent The map $\sigma$, defined in Lemma \ref{karpilovskysats}, will be referred to as the \emph{canonical action}.

\subsection{The Picard group of $\mathcal{R}_e$, $\Pic(\mathcal{R}_e)$}

We shall now give a brief description of the Picard group of $\mathcal{R}_e$ in a strongly graded ring $\mathcal{R} = \bigoplus_{g\in G} \mathcal{R}_g$. For more details we refer to \cite{CaenOyst}.

\begin{definition}[Invertible module]
Let $A$ be a ring. An $A$-bimodule $M$ is said to be \emph{invertible} if and only if there exists an $A$-bimodule $N$ such that $M \otimes_{A} N \cong A \cong N \otimes_{A} M$ as $A$-bimodules.
\end{definition}

Given a ring $A$, the Picard group of $A$, denoted $\Pic(A)$, is defined as the set of $A$-bimodule isomorphism classes of invertible $A$-bimodules, and the group operation is given by $\otimes_{A}$.

If $\mathcal{R} = \bigoplus_{g\in G} \mathcal{R}_g$ is a strongly $G$-graded ring, the homomorphism of $\rring_g \otimes_{\rring_e} \rring_h$ into $\rring_{gh}$ sending $r_g \otimes r_h$ into $r_g r_h$ for all $r_g \in \rring_g$ and $r_h \in \rring_h$, is an isomorphism of $\rring_e$-bimodules, for any $g,h\in G$ (see \cite[p.336]{Dade1982}). This implies that $\rring_g$ is an invertible $\rring_e$-bimodule for each $g\in G$. We may now define a group homomorphism $\psi : G \to \Pic(\rring_e)$, $g\mapsto [\rring_g]$, i.e. each $g\in G$ is mapped to the isomorphism class inside $\Pic(\rring_e)$ to which the invertible $\rring_e$-bimodule $\rring_g$ belongs.

\subsection{Crystalline graded rings}

We shall begin this section by recalling the definition of a crystalline graded ring. We would also like to emphasize that rings belonging to this class are in general not strongly graded.

\begin{definition}[Pre-crystalline graded ring]\label{precrystallinegradedring}
An associative and unital ring $\aalg$ is said to be \emph{pre-crystalline graded} if
\begin{enumerate}
	\item[(i)] there is a group $G$ (with neutral element $e$),
	\item[(ii)] there is a map $u : G \to \aalg, \quad g \mapsto u_g$ such that $u_e = 1_{\aalg}$ and $u_g \neq 0$ for every $g\in G$,
	\item[(iii)] there is a subring $\aalg_0 \subseteq \aalg$ containing $1_{\aalg}$,
\end{enumerate}
	such that the following conditions are satisfied:
\begin{enumerate}
	\item[(P1)] $\aalg = \bigoplus_{g\in G} \aalg_0 \, u_g$ ;
	\item[(P2)] For every $g\in G$, $u_g \, \aalg_0 = \aalg_0 \, u_g$ is a free left $\aalg_0$-module of rank one ;
	\item[(P3)] The decomposition in P1 makes $\aalg$ into a $G$-graded ring with $\aalg_0 = \aalg_e$.
\end{enumerate}
\end{definition}

\begin{lemma}[see \cite{CGR}]\label{sigmaalphamaps}
With notation and definitions as above:
\begin{enumerate}
	\item[{\rm (i)}] For every $g\in G$, there is a set map $\sigma_g : \aalg_0 \to \aalg_0$ defined by $u_g \, a = \sigma_g(a) \, u_g$ for $a \in \aalg_0$. The map $\sigma_g$ is a surjective ring morphism. Moreover, $\sigma_e = \identity_{\aalg_0}$.
	\item[{\rm (ii)}] There is a set map $\alpha : G \times G \to \aalg_0$ defined by $u_s \, u_t = \alpha(s,t) \, u_{st}$ for $s,t\in G$. For any triple $s,t,w \in G$ and $a\in \aalg_0$ the following equalities hold:
	\begin{eqnarray}
		\alpha(s,t)\alpha(st,w) = \sigma_s(\alpha(t,w)) \alpha(s,tw) \\
		\sigma_s(\sigma_t(a)) \alpha(s,t) = \alpha(s,t) \sigma_{st}(a)	\label{eq25}
	\end{eqnarray}
	\item[{\rm (iii)}] For every $g\in G$ we have $\alpha(g,e)=\alpha(e,g)=1_{\aalg_0}$ and $\alpha(g,g^{-1})=\sigma_g(\alpha(g^{-1},g))$.
\end{enumerate}
\end{lemma}

A pre-crystalline graded ring $\aalg$ with the above properties will be denoted by $\gcrystalline$ and each element of this ring is written as a sum $\sum_{g\in G} r_g \, u_g$ with coefficients $r_g \in \aalg_0$, of which only finitely many are non-zero. In \cite{CGR} it was shown that for pre-crystalline graded rings, the elements $\alpha(s,t)$ are normalizing elements of $\aalg_0$, i.e. $\aalg_0 \,\alpha(s,t)=\alpha(s,t)\,\aalg_0$ for each $s,t \in G$. For a pre-crystalline graded ring $\gcrystalline$, we let $S(G)$ denote the multiplicative set in $\aalg_0$ generated by $\{\alpha(g,g^{-1}) \, \mid \, g\in G\}$ and let $S(G \times G)$ denote the multiplicative set generated by $\{\alpha(g,h) \mid g,h \in G\}$.

\begin{lemma}[see \cite{CGR}]\label{torsionfree}
If $\aalg=\gcrystalline$ is a pre-crystalline graded ring, then the following assertions are equivalent:
\begin{enumerate}
	\item[{\rm (i)}] $\aalg_0$ is $S(G)$-torsion free.
 	\item[{\rm (ii)}] $\aalg$ is $S(G)$-torsion free.
	\item[{\rm (iii)}] $\alpha(g,g^{-1}) a_0 = 0$ for some $g\in G$ implies $a_0 = 0$.
	\item[{\rm (iv)}] $\alpha(g,h) a_0 =0$ for some $g,h\in G$ implies $a_0 =0$.
	\item[{\rm (v)}] $\aalg_0 \,u_g = u_g \,\aalg_0$ is also free as a right $\aalg_0$-module, with basis $u_g$, for every $g \in G$.
	\item[{\rm (vi)}] For every $g\in G$, $\sigma_g$ is bijective and hence a ring automorphism of $\aalg_0$.
\end{enumerate}
\end{lemma}

\begin{definition}[Crystalline graded ring]
A pre-crystalline graded ring $\gcrystalline$, which is $S(G)$-torsion free, is said to be a \emph{crystalline graded ring}.
\end{definition}

Note that $G$-crossed products are examples of crystalline graded rings.

\section{Ideals in strongly graded rings}\label{intersectionnewresult}
%\section{A theorem about intersections of ideals}
%\section{Ideal intersections in strongly $G$-graded rings}

In this section we shall improve some earlier results. We begin by making a slight generalization of Theorem \ref{stronglycommutantsnitt} (\cite[Corollary 3]{OinSilTheVav}). The following proof is based on the same technique as in \cite{OinSilTheVav}, but we will make it somewhat shorter by doing a proof by contra positivity.

\begin{theorem}\label{newzentersnitt}
If $\mathcal{R} = \bigoplus_{g\in G} \mathcal{R}_g$ is a strongly $G$-graded ring, then
\begin{displaymath}
	I\cap C_{\rring}(Z(\mathcal{R}_e)) \neq\{0\}
\end{displaymath}
for each nonzero ideal $I$ in $\mathcal{R}$.
\end{theorem}

\begin{proof}
Let $I$ be an ideal of $\rring$ such that $I \cap C_{\rring}(Z(\mathcal{R}_e)) =\{0\}$. If we can show that $I =  \{0\}$, then the desired conclusion follows by contra positivity. Take some $x = \sum_{g\in G} x_g \in I$. If $x\in I \cap C_{\rring}(Z(\mathcal{R}_e))$, then $x=0$ by the assumption. Therefore, assume that $x\in I \setminus C_{\rring}(Z(\mathcal{R}_e))$ and that $x$ is chosen such that $N= \# \supp(x) = \# \{g\in G\mid x_g \neq 0\}$ is as small as possible. Suppose that $N$ is positive. We are now seeking for a contradiction. Take some arbitrary $t \in \supp(x)$ and choose some $r_{t^{-1}} \in \rring_{t^{-1}}$ such that $x' = r_{t^{-1}} x \neq 0$ and $e\in \supp(x')$. It is always possible to choose such an $r_{t^{-1}}$. Indeed, if $1_{\rring} = \sum_{i=1}^{n_{t}} a_{t}^{(i)} b_{t^{-1}}^{(i)}$, as in \eqref{partitionofunity}, then $b_{t^{-1}}^{(i)} x_t$ must be nonzero for some $i \in \{1,\ldots,n_{t}\}$, for otherwise we would have $1_\rring x_t=0$ which would be contradictory (since $x_t \neq 0$). Note that $x' \in I$. Since $I \cap C_{\rring}(Z(\mathcal{R}_e)) =\{0\}$ we conclude that $x' \in I \setminus C_{\rring}(Z(\mathcal{R}_e))$. Take an arbitrary $a\in Z(\rring_e)$. Then $x''=ax'-x'a \in I$ but clearly $e\notin \supp(x'')$ and hence by the assumption on $N$ we get that $x''=0$. Since $a\in Z(\rring_e)$ was chosen arbitrarily we get $x' \in C_{\rring}(Z(\mathcal{R}_e))$ which is a contradiction. Therefore $N=0$ and hence $x=0$. Since $x \in I$ was arbitrarily chosen, we conclude that $I = \{0\}$.
\end{proof}

\begin{remark}
Note that $\rring_e \subseteq C_{\rring}(Z(\mathcal{R}_e))$. If $\rring_e$ is commutative, then clearly $\rring_e = Z(\rring_e)$ and we obtain Theorem \ref{stronglycommutantsnitt} as a special case of Theorem \ref{newzentersnitt}.
\end{remark}

%Consider a crystalline graded ring $\gcrystalline$. For each $g\in G$, one easily verifies that the automorhpism $\sigma_g : \aalg_0 \to \aalg_0$ restricts to an automorphism $\tilde{\sigma}_g : Z(\aalg_0) \to Z(\aalg_0)$. By \eqref{eq25} we conclude that $G$ acts as automorphisms of the ring $Z(\aalg_0)$. This observation makes it possible to carry out a proof analogously to the proof of Theorem \ref{newzentersnitt}. We shall omit the proof, but the conclusion is the following theorem which generalizes \cite[Corollary 8]{OinSil4}.

For a crystalline graded ring $\gcrystalline$, one may carry out a proof analogous to the proof of Theorem \ref{newzentersnitt}. We shall therefore omit the proof, but the conclusion is the following theorem which generalizes \cite[Corollary 8]{OinSil4}.

\begin{theorem}\label{newzentersnittcrystalline}
If $\aalg= \gcrystalline$ is a crystalline graded ring, then
\begin{displaymath}
	I\cap C_{\aalg}(Z(\aalg_0)) \neq\{0\}
\end{displaymath}
for each nonzero ideal $I$ in $\gcrystalline$.
\end{theorem}

The following theorem is a generalization of Theorem \ref{skewrings} (\cite[Theorem 3]{OinSil4}) and the proof makes use of the same idea as in \cite{OinSil4}. However, in this proof we make a crucial observation and are able to make use of an important map.

\begin{theorem}\label{skewringsnew}
Let $\mathcal{R}=\gskewring$ be a skew group ring with $\mathcal{R}_e$ commutative. The following two assertions are equivalent:
\begin{enumerate}
	\item[(i)] The ring $\mathcal{R}_e$ is a maximal commutative subring in $\mathcal{R}$.
	\item[(ii)] $I \cap \mathcal{R}_e \neq \{0\}$ for each nonzero ideal $I$ in $\mathcal{R}$.
\end{enumerate}
\end{theorem}

\begin{proof}
By Theorem \ref{newzentersnitt} (i) implies (ii) for the (strongly graded) skew group ring $\mathcal{R}$. We shall now show that (ii) implies (i). Suppose that $\mathcal{R}_e$ is not maximal commutative in $\mathcal{R}$. If we can show that there exists a nonzero ideal $I$ in $\mathcal{R}$, such that $I \cap \mathcal{R}_e =\{0\}$, then by contra positivity we are done. By the assumption there exists some $s\in G\setminus \{e\}$ and $r_s \in \mathcal{R}_e \setminus \{0\}$ such that $r_s \, \sigma_s(a) = r_s \, a$ for each $a\in \mathcal{R}_e$. Let us choose such an $r_s$ and let I be the twosided ideal in $\mathcal{R}$ generated by $r_s - r_s \, u_s$. The ideal $I$ is obviously nonzero, and furthermore it is spanned by elements of the form $a_g \, u_g \, (r_s - r_s \, u_s) \, a_h \, u_h$
where $g,h \in G$ and $a_g, a_h \in \mathcal{R}_e$. By commutativity of $\mathcal{R}_e$ and the properties of $r_s$ we may rewrite this expression.
\begin{eqnarray}
a_g \, u_g \, (r_s - r_s \, u_s) \, a_h \, u_h &=& a_g \, u_g \, (r_s \, a_h - \underbrace{r_s \, \sigma_s(a_h)}_{=r_s \, a_h} \, u_s) \, u_h \nonumber \\
&=& a_g \, u_g \, r_s \, a_h (1_{\mathcal{R}} - u_s) \, u_h \nonumber \\
&=& a_g \, \sigma_g(r_s \, a_h) \, u_g (1_{\mathcal{R}} - u_s) \, u_h \nonumber \\
&=& \underbrace{a_g \, \sigma_g(r_s \, a_h)}_{:=b} \, \, ( u_{gh} - u_{gsh}) \nonumber \\
&=& b \, u_{gh} - b \, u_{gsh} \label{element2}
\end{eqnarray}
It is now clear that each element of $I$ is a sum of elements of the form \eqref{element2}, where $b\in \mathcal{R}_e$ and $g,h\in G$. Define a map
\begin{displaymath}
	\epsilon : \gskewring \to \mathcal{R}_e, \quad \sum_{g\in G} a_g u_g \mapsto \sum_{g\in G} a_g .
\end{displaymath}
It is clear that $\epsilon$ is additive and one easily sees that $\epsilon$ is identically zero on $I$. Furthermore, $\epsilon\lvert_{\mathcal{R}_e}$, i.e. the restriction of $\epsilon$ to $\mathcal{R}_e$, is injective. Take an arbitrary $m\in I \cap \mathcal{R}_e$. Clearly $\epsilon(m)=0$ since $m\in I$ and by the injectivity of $\epsilon\lvert_{\mathcal{R}_e}$ we conclude that $m=0$. Hence $I \cap \mathcal{R}_e = \{0\}$. This concludes the proof.
\end{proof}

\begin{remark}
It is not difficult to see that the map $\epsilon$ is multiplicative if and only if the action $\sigma$ is trivial, i.e. $\gskewring$ is a group ring. In that situation the map $\epsilon$ is commonly referred to as the \emph{augmentation map}. However, note that the preceding proof does not require $\epsilon$ to be multiplicative.
\end{remark}

\section{The map $\psi : G \to \Pic(\rring_e)$ and simple strongly graded rings}\label{Simplicity}
%\section{Simplicity of strongly graded rings and the map $G \to \Pic(\rring_e)$}\label{Simplicity}

We begin by recalling a useful lemma.

\begin{lemma}[\cite{OinSilTheVav}]\label{annihilator}
Let $\mathcal{R} = \bigoplus_{g\in G} \mathcal{R}_g$ be a strongly $G$-graded ring. If $a\in \mathcal{R}$ is such that
\begin{displaymath}
	a \, \mathcal{R}_g = \{0\} \quad \text{or} \quad \mathcal{R}_g \, a = \{0\}
\end{displaymath}
for some $g\in G$, then $a=0$.
\end{lemma}

If we assume that $\mathcal{R}_e$ is maximal commutative in the strongly $G$-graded ring $\mathcal{R}$, then we can say the following about the canonical map $\psi : G \to \Pic(\mathcal{R}_e)$.

\begin{proposition}\label{picardinjective}
Let $\mathcal{R} = \bigoplus_{g\in G} \mathcal{R}_g$ be a strongly $G$-graded ring. If $\mathcal{R}_e$ is maximal commutative in $\mathcal{R}$, then the map $\psi : G \to \Pic(\mathcal{R}_e), g \mapsto [\mathcal{R}_g]$, is injective.
\end{proposition}

\begin{proof}
Let $\rring_e$ be maximal commutative in $\rring$. Suppose that $\psi : G \to \Pic(\mathcal{R}_e)$ is not injective. This means that we can pick two distinct elements $g,h \in G$ such that $\mathcal{R}_g \cong \mathcal{R}_h$ as $\mathcal{R}_e$-bimodules. Let $f : \mathcal{R}_g \to \mathcal{R}_h$ be a bijective $\mathcal{R}_e$-bimodule homomorphism. By our assumptions $\mathcal{R}_e =C_{\mathcal{R}}(\mathcal{R}_e)$ and hence we can use the map $\sigma : G\to \Aut(\mathcal{R}_e)$ defined by \eqref{defofsigma} to write
\begin{eqnarray}\label{rewrite}
\sigma_h(b) \, \underbrace{f(r_g)}_{\in \mathcal{R}_h} = f(r_g) \,  b = f(r_g \, b) = f(\sigma_g(b) \, r_g)
= \sigma_g(b) \, f(r_g)
\end{eqnarray}
for any $b\in \mathcal{R}_e$ and $r_g \in \mathcal{R}_g$. (It is important to note that $\sigma_g(b) \in \mathcal{R}_e$ since $b\in \mathcal{R}_e$.)
The map $f$ is bijective and in particular surjective. Hence, by \eqref{rewrite} we conclude that $(\sigma_h(b) - \sigma_g(b)) \, \mathcal{R}_h = \{0\}$ for any $b\in \mathcal{R}_e$. It follows from Lemma \ref{annihilator} that $\sigma_h(b) - \sigma_g(b)=0$ for any $b \in \mathcal{R}_e$. Hence $\sigma_g = \sigma_h$ in $\Aut(C_{\mathcal{R}}(\mathcal{R}_e)) = \Aut(\mathcal{R}_e)$ and we may write
\begin{displaymath}
\sigma_g = \sigma_h \quad \Longleftrightarrow \quad
	\sigma_{g^{-1}} \, \sigma_g = \sigma_{g^{-1}} \, \sigma_h \quad \Longleftrightarrow \quad \sigma_e = \sigma_{g^{-1}h} \quad \Longleftrightarrow \quad \identity_{\rring_e} = \sigma_{g^{-1}h}.
\end{displaymath}
Note that $g^{-1}h \neq e$ since $g\neq h$. This shows that the homogeneous component $\mathcal{R}_{g^{-1}h}$ ($\neq \mathcal{R}_e$) commutes with $\mathcal{R}_e$, and hence $\mathcal{R}_e$ is not maximal commutative. We have reached a contradiction and this shows that $\psi : G \to \Pic(\rring_e)$ is injective.
\end{proof}

% 1. J is obviously non-empty since 0 in I and B.
% 2. It is an additive group, since both I and B are additive groups.
% 3. För a \in B gäller a*J \in I eftersom varje element i J ligger i I som är ett ideal. Vidare gäller
% a*J \in B eftersom varje element i J ligger i B och B är en ring.  
% 4. Samma sak för högermultiplikation.

The following remark shows that the rings considered in Proposition \ref{maintheorem} are in fact $G$-crossed products.

\begin{remark}
Recall that a commutative and simple ring is a field. If $\mathcal{R} = \bigoplus_{g\in G} \mathcal{R}_g$ is a strongly $G$-graded ring and $\rring_e$ is a field, then $\rring$ is a $G$-crossed product. Indeed, for each $g\in G$, we have $\mathcal{R}_g \, \mathcal{R}_{g^{-1}} = \mathcal{R}_e$. Hence, by Lemma \ref{annihilator} we may choose some $a\in \rring_g \setminus \{0\}$ and $b\in \mathcal{R}_{g^{-1}} \setminus \{0\}$ such that $ab=c \neq 0$ in $\rring_e$. This means that $c$ is invertible in $\rring_e$ and hence $a$ is invertible in $\rring$, with (right) inverse $b c^{-1}$.
\end{remark}

%The following corollary is an obvious consequence of Theorem \ref{maintheorem}, but might be of practical importance.
%
%\begin{corollary}[of Theorem \ref{maintheorem}]
%Let $\mathcal{R} = \bigoplus_{g\in G} \mathcal{R}_g$ be a strongly $G$-graded ring where $\mathcal{R}_e$ is a field (i.e. $\rring$ is a $G$-crossed product). If $\mathcal{R}$ is not simple, then $\mathcal{R}_e$ is not maximal commutative in $\mathcal{R}$.
%\end{corollary}

The following proposition is a direct consequence of Theorem \ref{skewrings} and we shall therefore omit the proof.

\begin{proposition}\label{skewresult}
Let $\mathcal{R}=\gskewring$ be a skew group ring, where $\mathcal{R}_e$ is a field and $G$ is an abelian group. The following assertions are equivalent:
\begin{enumerate}
	\item[(i)] The subring $\mathcal{R}_e$ is maximal commutative in $\mathcal{R}$.
	\item[(ii)] $\mathcal{R}$ is a simple ring.
\end{enumerate}
\end{proposition}

%\begin{proof}
%Suppose that (i) holds. From Theorem \ref{maintheorem} it immediately follows that (ii) holds.
%%Suppose that (i) holds. Let $I$ be a non-zero ideal in $\mathcal{R}$. From Theorem \ref{skewrings} it follows that $I\cap \mathcal{R}_e \neq \{0\}$. The intersection $I \cap \mathcal{R}_e$ is clearly an ideal in $\mathcal{R}_e$. We conclude that $I \cap \mathcal{R}_e = \mathcal{R}_e$ since $\mathcal{R}_e$ is simple. Hence, $\mathcal{R}_e \subseteq I$ and in particular $1_\mathcal{R} \in I$, from which we conclude that $I=\mathcal{R}$ and hence $\mathcal{R}$ is a simple ring.
%Conversely, suppose that (ii) holds. From Theorem \ref{skewrings} we immediately conclude that $\mathcal{R}_e$ is maximal commutative in $\mathcal{R}$.
%\end{proof}

\begin{example}\label{LaurentGroupRing}
Consider the group ring $\mathcal{R} = \complex[\Z]$, which corresponds to the special case of a skew group ring with trivial action. The so called \emph{augmentation ideal}, which is the kernel, $\ker(\epsilon)$, of the augmentation map
\begin{displaymath}
	\epsilon : \C[\Z] \to \C, \quad \sum_{k \in \Z} c_k \, \overline{k} \mapsto \sum_{k \in \Z} c_k
\end{displaymath}
is a nontrivial ideal in $\mathcal{R}$ and hence $\mathcal{R}=\complex[\Z]$ is not a simple ring. This conclusion also follows directly from Proposition \ref{skewresult}. Indeed, $\mathcal{R} = \complex[\Z]$ is commutative and hence $\mathcal{R}_0 = \C$ is not maximal commutative in $\complex[\Z]$.
\end{example}

\subsection{Maximal commutativity of $\rring_e$ and injectivity of $\psi : G \to \Pic(\mathcal{R}_e)$}\label{maxcommpicard}
%\subsection{Maximal commutativity of $\rring_e$ and injectivity of the map $G \to \Pic(\mathcal{R}_e)$}\label{maxcommpicard}

The following proposition shows that in the case when $\mathcal{R}_e$ is assumed to be commutative, Theorem \ref{stronglysimple} is equivalent to Proposition \ref{maintheorem}.

\begin{proposition}\label{injectivemaxcommequiv}
Let $\mathcal{R} = \bigoplus_{g\in G} \mathcal{R}_g$ be a strongly $G$-graded ring. If $\rring_e$ is a field, then the following two assertions are equivalent:
\begin{itemize}
	\item[(i)] $\rring_e$ is maximal commutative in $\rring$.
	\item[(ii)] The map $\psi : G \to \Pic(\rring_e)$ is injective.
\end{itemize}
\end{proposition}

\begin{proof}
It follows from Proposition \ref{picardinjective} that (i) implies (ii).

To prove that (ii) implies (i), let us assume that $\mathcal{R}_e$ is not maximal commutative in $\mathcal{R}$. We want to show that $\psi$ is not injective and hence get the desired conclusion by contra positivity.

By our assumptions, there exists some nonzero element $r_g \in \mathcal{R}_g$, for some $g\neq e$, such that $r_g \, a = a \, r_g$ for all $a\in \mathcal{R}_e$. Consider the set $J = r_g \, \mathcal{R}_{g^{-1}} \subseteq \mathcal{R}_e$. Since $r_g$ commutes with $\mathcal{R}_e$ and $\rring_{g^{-1}}$ is an $\rring_e$-bimodule, $J$ is an ideal of $\mathcal{R}_e$ and as $r_g \, \mathcal{R}_{g^{-1}} \neq \{0\}$ (this follows from Lemma \ref{annihilator} since $r_g\neq 0$), we obtain $r_g \, \mathcal{R}_{g^{-1}} = \mathcal{R}_e$ since $\mathcal{R}_e$ is simple. Consequently, we conclude that there exists an $s_{g^{-1}} \in \rring_{g^{-1}}$ such that $ r_g \, s_{g^{-1}} = 1_{\rring}$. In a symmetrical way we get $\rring_{g^{-1}} \, r_g = \rring_e$ which yields $w_{g^{-1}} \, r_g = 1_{\rring}$  for some $w_{g^{-1}} \in \rring_{g^{-1}}$. From this we get $w_{g^{-1}} r_g s_{g^{-1}} = w_{g^{-1}}$ and hence $s_{g^{-1}}=w_{g^{-1}}$ yielding $r_g \, s_{g^{-1}} = s_{g^{-1}} \, r_g = 1_{\rring}$.

From the gradation we immediately conclude that $\mathcal{R}_e \, r_g \subseteq \mathcal{R}_g$ and $\mathcal{R}_g \, s_{g^{-1}} \subseteq \mathcal{R}_e$. By the equality $s_{g^{-1}} r_{g} = 1_{\rring}$ we get $\mathcal{R}_g \subseteq \mathcal{R}_e \, r_g$ and hence $\mathcal{R}_g = \mathcal{R}_e \, r_g$. Note that $r_g$ is invertible and hence a basis for the $\rring_e$-bimodule $\rring_e \, r_g$. This shows that $\mathcal{R}_g$ and $\mathcal{R}_e$ belong to the same isomorphism class in $\Pic(\mathcal{R}_e)$, and hence the morphism $\psi : G \to \Pic(\mathcal{R}_e)$ is not injective. This concludes the proof.
\end{proof}

\begin{remark}
The previous proof uses the same techniques as the proof of \cite[Theorem 3.4]{VanOystaeyen}.
\end{remark}

\section{$G$-simple subrings in crystalline graded rings}\label{GSimpleInCrystallineGradedRings}
%\section{$G$-simple subrings in simple crystalline graded rings}\label{GSimpleInCrystallineGradedRings}
%\section{Simple crystalline graded rings and $G$-simple subrings}

If $A$ is a ring and $\sigma : G \to \Aut(A)$ is a group action, then we say that an ideal $I$ in $A$ is \emph{$G$-invariant}
if $\sigma_g(I) \subseteq I$ for each $g\in G$. Note that it is equivalent to say that $\sigma_g(I)=I$ for each $g\in G$. If there are no nontrivial $G$-invariant ideals in $A$, then we say that $A$ is \emph{$G$-simple}.
%(This should not be confused with the term \emph{graded simple}!)
(Not to be confused with the term \emph{graded simple}!)

\begin{proposition}\label{CrystallineSimpleImplyGSimple}
Let $\gcrystalline$ be a crystalline graded ring, where $\aalg_0$ is commutative. If $\gcrystalline$ is a simple ring, then $\aalg_0$ is a $G$-simple ring (with respect to the action defined in Lemma \ref{sigmaalphamaps}).
\end{proposition}

\begin{proof}
Note that since $\aalg_0$ is commutative, the map $\sigma : G \to \Aut(\aalg_0)$ is a group homomorphism. Let $\gcrystalline$ be a simple ring, and $J$ an arbitrary nonzero $G$-invariant ideal in $\aalg_0$. One may verify that $J \Diamond_{\sigma}^{\alpha} G$ is a nonzero ideal of $\aalg_0 \Diamond_{\sigma}^{\alpha} G$. (This follows from the fact that for each $g\in G$, $\aalg_0 \, u_g$ is a free left $\aalg_0$-module with basis $u_g$.) Since $\aalg_0 \Diamond_{\sigma}^{\alpha} G$ is simple, we get $J \Diamond_{\sigma}^{\alpha} G = \aalg_0 \Diamond_{\sigma}^{\alpha} G$. Therefore $\aalg_0 \subseteq J \Diamond_{\sigma}^{\alpha} G$, and from the gradation it follows that
%we have
\begin{displaymath}
	\aalg_0 \subseteq J \subseteq \aalg_0
\end{displaymath}
and hence $\aalg_0 = J$, which shows that $\aalg_0$ is $G$-simple.
\end{proof}

\begin{corollary}\label{GcrossedSimpleImplyGSimple2}
Let $\mathcal{R} = \bigoplus_{g\in G} \mathcal{R}_g$ be a $G$-crossed product, where $\rring_e$ is commutative. If $\rring$ is a simple ring, then $\rring_e$ is a $G$-simple ring (with respect to the canonical action).
\end{corollary}

\begin{remark}\label{integralDomainNeedntBeSimple}
A field is automatically an integral domain. However, an integral domain need not be simple (and hence not a field). This is for example illustrated by the ring of polynomials in one variable $\C[x]$, which is an integral domain. Consider the subset $I=\{\sum_{k \in \Z_{\geq 0}} c_k \, x^k \, \mid \, c_k \in \C, \,\, c_0=0 \}$ of $\C[x]$. This is clearly a proper ideal in $\C[x]$, and hence $\C[x]$ is not simple.
\end{remark}

\begin{example}\label{FirstWeylAlgebraExample}
It is well-known that the first Weyl algebra $\aalg = \frac{\complex \langle x,y \rangle}{(xy-yx-1)}$ is simple. The first Weyl algebra is an example of a crystalline graded ring, with $G=(\Z,+)$ and $\aalg_e=\aalg_0 = \complex[xy]$ (see e.g. \cite{CGR} for details).
By Remark \ref{integralDomainNeedntBeSimple} the ring $\aalg_0$ is not simple. However, by Proposition \ref{CrystallineSimpleImplyGSimple} we conclude that $\aalg_0=\complex[xy]$ is $\Z$-simple. As a side remark we should also mention that one can show that $\aalg_0$ is maximal commutative in $\gcrystalline$.
\end{example}

\section{$G$-simple subrings in strongly $G$-graded rings}\label{GSimpleSubringsStronglyGraded}
%\section{$G$-simple subrings in simple strongly $G$-graded rings}\label{GSimpleSubringsStronglyGraded}
%\section{Simple strongly graded rings and $G$-simple subrings}
%%\subsection{Simplicity of strongly graded rings and $G$-simplicity of its subrings}

In this section we shall describe how simplicity of a strongly $G$-graded ring $\rring = \bigoplus_{g\in G} \rring_g$ is related to $G$-simplicity of the subrings $Z(\rring_e)$ and $C_{\rring}(\rring_e)$. We begin by recalling the following basic lemma which is easily verified.
\begin{lemma}\label{idealproj}
Let $A$ be a ring containing a subring $B$. If $I$ is an ideal in $A$, then $J= I\cap B$ is an ideal in $B$.
\end{lemma}

Using this we obtain the following.

\begin{lemma}\label{Gsimpleimpliesproperintersection}
Let $\mathcal{R} = \bigoplus_{g\in G} \mathcal{R}_g$ be a strongly $G$-graded ring. If $Z(\rring_e)$ is a $G$-simple ring (with respect to the canonical action), then no proper ideal of $\rring$ intersects $Z(\rring_e)$ nontrivially.
\end{lemma}

\begin{proof}
Let $Z(\rring_e)$ be $G$-simple and $I$ an ideal of $\rring$ which intersects $Z(\rring_e)$ nontrivially. By Lemma \ref{idealproj}, $J=I \cap Z(\rring_e)$ is an ideal in $Z(\rring_e)$. For any $x \in J$ and every $g\in G$, we have $\sigma_g(x) = \sum_{i=1}^{n_g} a_{g}^{(i)} \, x \, b_{g^{-1}}^{(i)} \in I \cap Z(\rring_e) = J$. This shows that $J$ is a $G$-invariant ideal of $Z(\rring_e)$. By assumption $J$ is nonzero and hence $J=Z(\rring_e)$. In particular this shows that $1_\rring \in J$, from which we get $1_\rring \in I$ and hence $\rring = I$.
\end{proof}

If $\rring_e$ is commutative, then clearly $\rring_e = Z(\rring_e)$ and hence we have an action $\sigma : G \to \Aut(\rring_e)$. The following proposition generalizes Corollary \ref{GcrossedSimpleImplyGSimple2}.

\begin{proposition}\label{simpleimpliesGsimple}
Let $\mathcal{R} = \bigoplus_{g\in G} \mathcal{R}_g$ be a strongly $G$-graded ring, where $\rring_e$ is commutative. If $\rring$ is a simple ring, then $\rring_e$ is a $G$-simple ring (with respect to the canonical action).
\end{proposition}

\begin{proof}
Let $J$ be an arbitrary nonzero $G$-invariant ideal in $\rring_e$. Consider the subset $J \rring$ of $\rring$. (Note that $J \rring$ denotes not only products of the form $jr$, $j\in J$, $r\in \rring$, but finite sums of such elements!)
It is easy to see that $J \rring$ is a right ideal in $\rring$, and from the fact that $J$ is a $G$-invariant ideal in $\rring_e$ we conclude that $J \rring$ is also a left ideal. Indeed, for $g,h\in G$ and $c\in J$, $r_h\in \rring_h$, $s_g \in \rring_g$ we have $s_g \, c \, r_h = \sigma_g(c) \, s_g \, r_h \in J\rring$.
Furthermore, $\rring$ is unital and hence $J \rring$ must be nonzero. The ring $\rring$ is simple and therefore we conclude that $J \rring = \rring$. In particular we see that $\rring_e \subseteq J \rring$. From the strong gradation we conclude that
\begin{displaymath}
	\rring_e \subseteq J \rring_e \subseteq J \subseteq \rring_e
\end{displaymath}
and hence $J = \rring_e$. This shows that $\rring_e$ is $G$-simple.
\end{proof}

%\begin{proposition}\label{simpleimpliesGsimple}
%Let $\mathcal{R} = \bigoplus_{g\in G} \mathcal{R}_g$ be a strongly $G$-graded ring where $\rring_e$ is a domain. If $\rring$ is simple, then $Z(\rring_e)$ is $G$-simple (with respect to the usual action).
%\end{proposition}
%
%\begin{proof}
%Suppose that $J$ is a nonzero $G$-invariant ideal of $Z(\rring_e)$. Consider the subset $J \rring = \bigoplus_{g\in G} J \rring_g$ of $\rring$. It is easy to see that $J \rring$ is a right ideal in $\rring$, and from the fact that $J$ is a $G$-invariant ideal of $Z(\rring_e)$ we conclude that $J \rring$ is also a left ideal.
%
%Indeed, for $c \in J$, $\sum_{h\in G} r_h \in \rring$ and any $\sum_{g \in G} s_g \in \rring$ we have
%\begin{eqnarray*}
%	\sum_{g \in G} s_g  \left(c   \sum_{h\in G} r_h	\right) &=& \underbrace{\left( \sum_{g \in G} \sigma_g(c) s_g \right)}_{\in J\rring}  \, \left(   \sum_{h\in G} r_h	\right) \in J\rring
%\end{eqnarray*}
%
%Furthermore, $\rring$ is unital and therefore $J \rring$ must be nonzero. The ring $\rring$ is simple and therefore we conclude that $J \rring = \rring$. In particular we see that $Z(\rring_e) \subseteq J \rring$. From the strong gradation we conclude that
%\begin{displaymath}
%	Z(\rring_e) \subseteq \rring_e \subseteq J \rring_e \substack{? \\ \subseteq} J \subseteq Z(\rring_e)
%\end{displaymath}
%and hence $J = Z(\rring_e)$. This shows that $Z(\rring_e)$ is $G$-simple.
%HÄR FINNS FRÅGETECKEN ATT REDA UT!!!!!!  I INKLUSIONERNA OVAN!!! VAD ÄR IDEAL I VAD OCH EXAKT VAD GER GRADERINGEN???
%\end{proof}

We shall now direct our attention to the subring $C_{\rring}(\rring_e)$ of $\rring$. Note that the following lemma does not require $\rring_e$ to be commutative.

\begin{lemma}\label{helpskit}
Let $\mathcal{R} = \bigoplus_{g\in G} \mathcal{R}_g$ be a strongly $G$-graded ring. If $C_{\rring}(\rring_e)$ is a $G$-simple ring (with respect to the canonical action), then no proper ideal of $\rring$ intersects $C_{\rring}(\rring_e)$ nontrivially.
\end{lemma}

\begin{proof}
Let $C_{\rring}(\rring_e)$ be $G$-simple and $I$ an ideal of $\rring$ which intersects $C_{\rring}(\rring_e)$ nontrivially. 
By Lemma \ref{idealproj}, $J=I \cap C_{\rring}(\rring_e)$ is an ideal in $C_{\rring}(\rring_e)$. For any $x \in J$ and every $g\in G$, we have $\sigma_g(x) = \sum_{i=1}^{n_g} a_{g}^{(i)} \, x \, b_{g^{-1}}^{(i)} \in I \cap C_{\rring}(\rring_e) = J$.
%The rest of the proof is analogous to the proof of Lemma \ref{Gsimpleimpliesproperintersection}.
This shows that $J$ is a $G$-invariant ideal of $C_{\rring}(\rring_e)$. By assumption $J$ is nonzero and hence $J=C_{\rring}(\rring_e)$. In particular this shows that $1_\rring \in J$, from which we see that $1_\rring \in I$ and hence $\rring = I$.
\end{proof}

\begin{corollary}
Let $\mathcal{R} = \bigoplus_{g\in G} \mathcal{R}_g$ be a semiprime PI-ring which is strongly $G$-graded. If either $Z(\rring_e)$ or $C_{\rring}(\rring_e)$ is a $G$-simple ring (with respect to the canonical action), then $\rring$ is a simple ring.
\end{corollary}

\begin{proof}
Let $I$ be a nonzero ideal in $\rring$. It follows from \cite[Theorem 2]{Rowen} that $I \cap Z(\rring) \neq \{0 \}$. Clearly $Z(\rring) \subseteq Z(\rring_e) \subseteq C_{\rring}(\rring_e)$ and hence by Lemma \ref{Gsimpleimpliesproperintersection} or \ref{helpskit} we conclude that $I = \rring$.
\end{proof}

As we shall see Theorem \ref{stronglyGgradedSimpleNecSuff} requires $\rring_e$ not only to be commutative, but maximal commutative in $\rring$. We begin with the following proposition which applies to the more general situation when $\rring_e$ is not necessarily maximal commutative in $\rring$.

\begin{proposition}\label{GSimpleCommutantSimple}
Let $\mathcal{R} = \bigoplus_{g\in G} \mathcal{R}_g$ be a strongly $G$-graded ring, where $\rring_e$ is commutative. If $C_{\rring}(\rring_e)$ is a $G$-simple ring (with respect to the canonical action), then $\rring$ is a simple ring.
\end{proposition}

\begin{proof}
Let $I$ be an arbitrary nonzero ideal of $\rring$. Since $\rring_e$ is commutative it follows from Theorem \ref{stronglycommutantsnitt} that $I \cap C_{\rring}(\rring_e) \neq \{0\}$. By Lemma \ref{helpskit} we conclude that $I=\rring$ and hence $\rring$ is a simple ring.
\end{proof}

By combining Proposition \ref{simpleimpliesGsimple} and Proposition \ref{GSimpleCommutantSimple} we get the following theorem.

\begin{theorem}\label{stronglyGgradedSimpleNecSuff}
Let $\mathcal{R} = \bigoplus_{g\in G} \mathcal{R}_g$ be a strongly $G$-graded ring. If $\rring_e$ is maximal commutative in $\rring$, then the following two assertions are equivalent:
\begin{enumerate}
	\item[(i)] $\rring_e$ is a $G$-simple ring (with respect to the canonical action).
	\item[(ii)] $\rring$ is a simple ring.
\end{enumerate}
\end{theorem}

One should note that Proposition \ref{GSimpleCommutantSimple} and Theorem \ref{stronglyGgradedSimpleNecSuff} are more general than Theorem \ref{stronglysimple} in the sense that $\rring_e$ is not required to be simple. On the other hand, this did not come for free. We had to make an additional assumption on $\rring_e$, namely that it was commutative.

\begin{remark}
Note that Theorem \ref{stronglyGgradedSimpleNecSuff} especially applies to $G$-crossed products.
\end{remark}

One may think that for a simple strongly graded ring $\mathcal{R} = \bigoplus_{g\in G} \mathcal{R}_g$ where $\rring_e$ is commutative and $G$-simple, this would imply that $\rring_e$ would be maximal commutative in $\rring$. In general this is not true, as the following example shows.

\begin{example}\label{complexnumbersexample}
Consider the field of complex numbers $\C = \R \rtimes^{\alpha} \Z_2$ as a $\Z_2$-graded twisted group ring (see e.g. \cite{OinSilAGMFGBG} for details). Clearly $\C$ is simple as is $\R$. Hence $\R$ is also $\Z_2$-simple, but it is not maximal commutative in $\C$.
\end{example}

The purpose of the following example is to present a strongly group graded ring which is not a crossed product, and to identify a $G$-simple subring. 

\begin{example}[A strongly group graded, noncrossed product, matrix ring]\label{3by3matricesZ2grading}
Let $\rring=M_3(\C)$ denote the ring of $3 \times 3$-matrices over $\C$. By putting
\begin{displaymath}
\rring_0 =
\left( 
\begin{array}{ccc}
\C & \C & 0 \\
\C & \C & 0 \\
0 & 0 & \C
\end{array}
\right)
\quad \text{and} \quad
\rring_1 =
\left( 
\begin{array}{ccc}
0 & 0 & \C \\
0 & 0 & \C \\
\C & \C & 0
\end{array}
\right)
\end{displaymath}
one may verify that this defines a strong $\mathbb{Z}_2$-gradation on $\rring$. However, note that $\rring$ is \underline{not} a crossed product with this grading since the homogeneous component $\rring_1$ does not contain any invertible elements of $M_3(\C)$! A simple calculation yields
\begin{displaymath}
	Z(\rring_0) = \left\{ \left( \begin{array}{ccc}
a & 0 & 0 \\
0 & a & 0 \\
0 & 0 & b
\end{array}\right) \Bigg| \quad a,b\in \C \right\}
\end{displaymath}
and in fact one may verify that $C_{\rring}(\rring_0)=Z(\rring_0$). In order to define an action $\sigma : \Z_2 \to \Aut(Z(\rring_0))$ we need to make a decomposition of the identity matrix $I=1_\rring$, in accordance with \eqref{partitionofunity}.
Let $E_{i,j}$ denote the $3 \times 3$-matrix which has a $1$ in position $(i,j)$ and zeros everywhere else. The decomposition in $\rring_0$ is trivial, but in $\rring_1$ we may for example choose
\begin{displaymath}
	I = E_{1,3}E_{3,1}+E_{2,3}E_{3,2}+E_{3,2}E_{2,3}.
\end{displaymath}
From these decompositions we are now able to define the map $\sigma : \Z_2 \to \Aut(Z(\rring_0))$. By looking at the ideal generated by $E_{3,1}$ it is clear that $Z(\rring_0)$ is not simple. However, there are only two nontrivial ideals in $Z(\rring_0)$ and one easily checks that neither of them is invariant under $\sigma_1$. This shows that for our simple ring $M_3(\C)$, the subring $Z(\rring_0) = C_\rring(\rring_0)$ is in fact $\Z_2$-simple .
\end{example}

\begin{remark}
In many examples of graded rings $A = \bigoplus_{g\in G} A_g$, $A_e$ is commutative. In that situation we always have $A_e \subseteq C_A(A_e)$. However, in Example \ref{3by3matricesZ2grading}, $\rring_0$ is not commutative and we in fact get that $C_{\rring}(\rring_0)$ coincides with $Z(\rring_0)$ which is smaller than $\rring_0$. 
\end{remark}

Proposition \ref{simpleimpliesGsimple} shows that in a simple strongly graded ring $\rring$ where $\rring_e$ is commutative, we automatically have that $\rring_e = Z(\rring_e)$ is $G$-simple. In Proposition \ref{GSimpleCommutantSimple} we saw that for a strongly graded ring $\rring$ where $\rring_e$ is commutative, $G$-simplicity of $C_{\rring}(\rring_e)$ implies simplicity of $\rring$. After seing Example \ref{3by3matricesZ2grading} it is tempting to think that the converse is also true (even for noncommutative $\rring_e$), i.e. that simplicity of $\rring$ always gives rise to $G$-simple subrings. The natural questions are: 
\begin{enumerate}
	\item If $\rring$ is strongly group graded and simple, is $C_{\rring}(\rring_e)$ necessarily $G$-simple?
	\item If $\rring$ is strongly group graded and simple, is $Z(\rring_e)$ necessarily $G$-simple?
\end{enumerate}
At the moment we do not know how to prove this in the most general situation, and we can not find any counter example either.

\begin{remark}
	Note that, if $\rring$ is commutative, then it is trivial to verify that the answers to both questions are affirmative. If $\rring_e$ is maximal commutative, then by Theorem \ref{stronglyGgradedSimpleNecSuff} we again conclude that the answers to both questions are affirmative. Also, if $\rring_e$ is commutative it follows by Proposition \ref{simpleimpliesGsimple} that the answer to question nr. 2 is affirmative. The case that remains to be investigated is that of a noncommutative ring $\rring$ where $\rring_e$ is not maximal commutative (we may not even assume it to be commutative) in $\rring$.
\end{remark}

\subsection{Simplicity of skew group rings}\label{SimplicityOfSkewGroupRings}

From Example \ref{complexnumbersexample} we learnt that simplicity of a strongly graded ring $\rring$ does not immediately imply maximal commutativity of the neutral component $\rring_e$. However, for skew group rings there is in fact such an implication, as the following theorem shows.

\begin{theorem}\label{skewringssimple}
Let $\mathcal{R}=\gskewring$ be a skew group ring with $\mathcal{R}_e$ commutative. The following two assertions are equivalent:
\begin{enumerate}
	\item[(i)] $\mathcal{R}_e$ is a maximal commutative and $G$-simple subring in $\mathcal{R}$.
	\item[(ii)] $\mathcal{R}=\gskewring$ is a simple ring.
\end{enumerate}
\end{theorem}

\begin{proof}
By Theorem \ref{stronglyGgradedSimpleNecSuff}, (i) implies (ii). Suppose that (ii) holds. It follows by Theorem \ref{skewringsnew} that $\rring_e$ is maximal commutative in $\rring$ and by Proposition \ref{simpleimpliesGsimple} we conclude that $\rring_e$ is $G$-simple. This concludes the proof.
\end{proof}

It follows from \cite[Corollary 10]{OinSil} that the assumptions made in \cite[Corollary 2.1]{Crow} force the coefficient ring to be maximal commutative in the skew group ring, and by the assumptions made in \cite[Theorem 2.2]{Crow} the same conclusion follows by \cite[Proposition 2.2]{Crow} and \cite[Corollary 7]{OinSil}. This shows that Theorem \ref{skewringssimple} is a generalization of \cite[Corollary 2.1]{Crow} and \cite[Theorem 2.2]{Crow}.

%Once one realizes that the assumptions made in \cite[Corollary 2.1]{Crow} respectively \cite[Theorem 2.2]{Crow} force the coefficient ring to be maximal commutative in the skew group ring, it is clear that Theorem \ref{skewringssimple} is a generalization of those two results. It follows from \cite[Corollary 10]{OinSil} that the assumptions made in \cite[Corollary 2.1]{Crow} force the coefficient ring to be maximal commutative in the skew group ring follows by, and for \cite[Theorem 2.2]{Crow} this follows by \cite[Proposition 2.2]{Crow} and \cite[Corollary 7]{OinSil}.

\begin{remark}
Note that, in Theorem \ref{skewringssimple}, the implication from (i) to (ii) holds in much greater generality. Indeed, it holds for any strongly graded ring.
%The converse, however, requires more assumptions.
\end{remark}

%\begin{example}[The quantum torus]
%...JGLTA 1 eller Crossed Product-Like and Pre-Crystalline Graded Rings
%\end{example}

A majority of the objects studied in \cite{SSD1,SSD2,SSD3} satisfy the conditions of Theorem \ref{skewringssimple} and hence it applies. We shall show one such example.

\begin{example}[Skew group algebras associated to dynamical systems]\label{torrtsystem}
Let $h : X \to X$ be a bijection on a nonempty set $X$, and $A \subseteq \C^X$ an algebra of functions, such that if $f \in A$ then $f \circ h \in A$ and $f \circ h^{-1} \in A$. Let $\tilde{h} : \Z \to \Aut(A)$ be defined by $\tilde{h}_n : f \mapsto f \circ h^{\circ (n)}$ for $f \in A$ and $n\in \Z$. We now have a $\Z$-crossed system (with trivial $\tilde{h}$-cocycle) and we may define the skew group algebra $A \rtimes_{\tilde{h}} \Z$. For more details we refer to the papers \cite{SSD1,SSD2,SSD3}, in which this construction has been studied thoroughly.
\end{example}

By Theorem \ref{skewringssimple} we get the following corollary, since $\C^X$ is commutative.

%The algebra $\C^X$ is commutative and hence Theorem \ref{skewringssimple} gives the following corollary.

\begin{corollary}
Following Example \ref{torrtsystem}, let $A \rtimes_{\tilde{h}} \Z$ be the skew group algebra associated to a dynamical system $(X,h)$. The following assertions are equivalent:
\begin{enumerate}
	\item[(i)] $A \rtimes_{\tilde{h}} \Z$ is a simple algebra.
	\item[(ii)] $A$ is a maximal commutative and $\Z$-simple subalgebra in $A \rtimes_{\tilde{h}} \Z$.
\end{enumerate}
\end{corollary}

\section{Three different proofs of Proposition \ref{maintheorem}}\label{ProofSection}

We shall now give three different proofs of Proposition \ref{maintheorem}. 

\begin{proof}[First proof based on Theorem \ref{stronglycommutantsnitt}]
Let $I$ be a nonzero ideal in $\mathcal{R}$. Our assumptions together with Lemma \ref{idealproj} and Theorem \ref{stronglycommutantsnitt} ensure that $J=I \cap R_{e}$ is a nonzero ideal in $\mathcal{R}_e$. We have $J=\mathcal{R}_e$, since $\mathcal{R}_e$ is simple, and this yields $\mathcal{R}_e \subseteq I$. Recall that $1_{\mathcal{R}} \in \mathcal{R}_e \subseteq I$, and hence $I=\mathcal{R}$. This shows that $\mathcal{R}$ is simple.
\end{proof}

\begin{proof}[Second proof based on Theorem \ref{stronglysimple}]
It follows by our assumptions and Proposition \ref{picardinjective} that the map $\psi : G \to \Pic(\mathcal{R}_e), g \mapsto [\mathcal{R}_g]$, is injective. The neutral component $\mathcal{R}_e$ is assumed to be simple, and hence by Theorem \ref{stronglysimple} $\mathcal{R}$ is simple.
\end{proof}

\begin{proof}[Third proof based on Proposition \ref{GSimpleCommutantSimple}]
By our assumptions $\rring_e=Z(\rring_e)$ $=C_{\rring}(\rring_e)$ is simple and hence in particular $G$-simple. It now follows immediately from Proposition \ref{GSimpleCommutantSimple} that $\rring$ is simple.
\end{proof}

\section{Application: $\Z$-graded algebraic crossed products associated to topological dynamical systems}\label{topdynsys}
%\section{The $\Z$-graded crossed product algebra associated to a topological dynamical system}

Let $(X,h)$ be a topological dynamical system, i.e. $X$ is a compact Hausdorff space and $h : X \to X$ is a homeomorphism. The algebra of complex-valued continuous functions on $X$, where addition and multiplication is defined pointwise, is denoted by $C(X)$. Define a map
\begin{displaymath}
	\tilde{h} : \Z \to \Aut( C(X)), \quad \tilde{h}_n(f) = f \circ h^{\circ(n)}, \quad f \in C(X)
\end{displaymath}
and let $C(X) \rtimes_{\tilde{h}} \Z$ be the algebraic crossed product\footnote{In ring theory literature this would be referred to as a \emph{skew group algebra}, but here we adopt the terminology used in \cite{SSD1,SSD2,SSD3} which comes from the $C^*$-algebra literature. Note however, that this is not a $C^*$-crossed product, but an algebraic crossed product.} associated to our dynamical system. Recall that elements of $C(X) \rtimes_{\tilde{h}} \Z$ are written as formal sums $\sum_{n\in \Z} f_n \, u_n$, where all but a finite number of $f_n \in C(X)$, for $n\in\Z$, are nonzero. The multiplication in $C(X) \rtimes_{\tilde{h}} \Z$ is defined as the bilinear extension of the rule
\begin{displaymath}
	(f_n \, u_n) (g_m \, u_m) = f_n \, \tilde{h}_n(g_m) \, u_{n+m}
\end{displaymath}
for $n,m \in \Z$ and $f_n,g_m \in C(X)$. We now define the following sets:
\begin{eqnarray*}
\Per^n(h) &=& \left\{ x \in X \, \mid \, h^{\circ (n)}(x)=x \right\}, \quad n \in \Z \\
\Per(h) &=& \bigcup_{n\in \Z} \Per^n(h) \\
\Aper(h) &=& X \setminus \Per(h)
\end{eqnarray*}

Elements of $\Aper(h)$ are referred to as \emph{aperiodic points} of the topological dynamical system $(X,h)$. By Urysohn's lemma, $C(X)$ separates points of $X$ and hence by \cite[Corollary 3.4]{SSD1} we get the following.

\begin{lemma}\label{commutantTopDynSys}
The commutant of $C(X)$ in $\rring = C(X) \rtimes_{\tilde{h}} \Z$ is given by
\begin{displaymath}
	C_{\rring}(C(X)) = \left\{ \sum_{n\in \Z} f_n \, u_n \,\, \Big| \,\, \supp(f_n) \subseteq \Per^n(h) , \quad f_n \in C(X), \quad n\in \Z \right\}.
\end{displaymath}
\end{lemma}

The topological dynamical system $(X,h)$ is said to be \emph{topologically free} if and only if $\Aper(h)$ is dense in $X$. Using topological properties of our (completely regular) space $X$ together with Lemma \ref{commutantTopDynSys}, one can prove the following.

\begin{lemma}\label{maxcommequivtopfree}
$C(X)$ is maximal commutative in $C(X) \rtimes_{\tilde{h}} \Z$ if and only if $(X,h)$ is topologically free.
\end{lemma}

%One can show that there is a 1-1 correspondence between closed subsets of $X$ and ideals of $C(X)$.

If $I$ is an ideal of $C(X)$ then we denote
\begin{displaymath}
	\supp(I) = \bigcup_{f\in I} \supp(f)
\end{displaymath}
where $\supp(f) = \overline{ \{x\in X \, \mid \, f(x) \neq 0  \} }$ for $f\in C(X)$. Note that a subset $S \subseteq X$ is $\Z$-invariant if and only if $h(S) = S$.

\begin{lemma}\label{ZsimpleZinvSubsets}
$C(X)$ is $\Z$-simple if and only if there are no nonempty proper $h$-invariant closed subsets of $X$.
\end{lemma}

\begin{proof}
Suppose that $C(X)$ is not $\Z$-simple. Then there exists some proper nonzero ideal $I \subsetneq C(X)$ such that $\supp(I) \neq \emptyset$ is a proper $h$-invariant closed subset of $X$.
Conversely, suppose that there exists some nonempty proper $h$-invariant closed subset $S \subsetneq X$. Let $B \subseteq C(X)$ be set of functions which vanish outside $S$. Clearly $B$ is a proper nonzero $\Z$-invariant ideal of $C(X)$ and hence $C(X)$ is not $\Z$-simple. 
\end{proof}

\begin{definition}
A topological dynamical system $(X,h)$ is said to be \emph{minimal} if each orbit of the dynamical system is dense in $X$.
\end{definition}

Note that a topological dynamical system $(X,h)$ is minimal if and only if there are no nonempty proper $h$-invariant closed subsets of $X$.

\begin{remark}\label{minimalityimpliesfreeness}
If $X$ is infinite and $(X,h)$ is \emph{minimal}, then $(X,h)$ is automatically \emph{topologically free}. Indeed, take an arbitrary $x\in X$ and suppose that it is perodic. By minimality, the orbit of $x$ which by periodicity is finite, must be dense in $X$. This is a contradiction, since $X$ is Hausdorff, and hence each $x\in X$ is aperiodic.
\end{remark}

\begin{theorem}\label{algebraicanalogue}
If $(X,h)$ is a topological dynamical system with $X$ infinite, then the following assertions are equivalent:
\begin{enumerate}
	\item[(i)] $C(X) \rtimes_{\tilde{h}} \Z$ is a simple algebra.
	\item[(ii)] $C(X)$ is maximal commutative in $C(X) \rtimes_{\tilde{h}} \Z$ and $C(X)$ is $\Z$-simple. 
	\item[(iii)] $(X,h)$ is a minimal dynamical system.
\end{enumerate}
\end{theorem}

\begin{proof}
(i) $\Longleftrightarrow$ (ii): This follows from Theorem \ref{skewringssimple}.\\
(iii) $\Rightarrow$ (ii): Let $(X,h)$ be  minimal. By Remark \ref{minimalityimpliesfreeness} $(X,h)$ is topologically free and by Lemma \ref{maxcommequivtopfree} this implies that $C(X)$ is maximal commutative in $C(X) \rtimes_{\tilde{h}} \Z$. Furthermore, since $(X,h)$ is minimal there is no nonempty proper $h$-invariant closed subset of $X$ and hence by Lemma \ref{ZsimpleZinvSubsets} it follows that $C(X)$ is $\Z$-simple.\\
(ii) $\Rightarrow$ (iii): Suppose that $(X,h)$ is not minimal. Then there exists some nonempty proper $h$-invariant closed subset of $X$ and by Lemma \ref{ZsimpleZinvSubsets} $C(X)$ is not $\Z$-simple.
\end{proof}

For $C^*$-crossed product algebras associated to topological dynamical systems the analogue of the above theorem, Theorem \ref{simplicityCstar}, is well-known (see e.g. \cite{ArchSpiel}, \cite{Power} or \cite[Theorem 4.3.3]{TomiyamaBook}). C. Svensson and J. Tomiyama recently proved that the analogue of the above theorem also holds for Banach $*$-algebra crossed products (in $L^1$-norm) associated to topological dynamical systems (see \cite[Theorem 4.2]{ChristianTomiyama2}).

\begin{example}[Finite single orbit dynamical systems]
Suppose that $X=\{x,h(x),h^{\circ (2)}(x),\ldots,h^{\circ (p-1)}(x)\}$ consists of a finite $h$-orbit of order $p$, where $p$ is a positive integer. One can then show that $C(X) \rtimes_{\tilde{h}} \Z \cong M_p(\C[t,t^{-1}])$, i.e. the skew group algebra associated to our dynamical system is isomorfic (as a $\C$-algebra) to the algebra of $p \times p$-matrices over the ring of Laurent polynomials over $\C$. Indeed, let $\pi : C(X) \rtimes_{\tilde{h}} \Z \to M_p(\C[t,t^{-1}])$ be the $\C$-algebra morphism defined by
\begin{displaymath}
	\pi(f) =
	\left( \begin{array}{cccc}
	f(x)	&	0	& \ldots	& 0 \\
	0	&	f \circ h(x)	& \ldots	& 0 \\
	\vdots &	\vdots	& \ddots	& \vdots \\
	0	&	0	& \ldots	& f \circ h^{\circ(p-1)}(x)
	\end{array} \right)
\end{displaymath}
for $f\in C(X)$, and
\begin{displaymath}
	\pi(u_1) =
	\left( \begin{array}{ccccc}
	0	&	0	& \ldots	& 0 & t \\
	1	&	0	& \ldots	& 0 & 0 \\
	0	&	1	& \ldots	& 0 & 0 \\
	\vdots &	\vdots	& \ddots	& \vdots & \vdots \\
	0	&	0	& \ldots	& 1 & 0
	\end{array} \right).
\end{displaymath}
Calculating, one sees that
\begin{eqnarray*}
	&& \pi\left( \sum_{n\in \Z} f_n \, u_n \right) = \\
	&& \left( \begin{array}{ccc}
	\sum_{n\in \Z} f_{np}(x) \, t^n	&	 \ldots	& \sum_{n\in \Z} f_{(n-1)p+1}(x) \, t^n \\
		\sum_{n\in \Z} f_{np+1}(h(x)) \, t^n	&	 \ldots	& \sum_{n\in \Z} f_{(n-1)p+2}(h(x)) \, t^n \\
			\sum_{n\in \Z} f_{np+2}(h^{\circ(2)}(x)) \, t^n	&	 \ldots	& \sum_{n\in \Z} f_{(n-1)p+3}(h^{\circ(2)}(x)) \, t^n \\
			\vdots & \vdots & \vdots \\
			\sum_{n\in \Z} f_{(n+1)p-1}(h^{\circ(p-1)}(x)) \, t^n	&	 \ldots	& \sum_{n\in \Z} f_{np}(h^{\circ(p-1)}(x)) \, t^n \\
	\end{array} \right)
\end{eqnarray*}
and by looking at the above matrix row by row, it is straightforward to verify that $\pi$ is bijective (see \cite{ChristianTomiyama,TomiyamaPaper} for a similar isomorphism of $C^*$-algebras).

Clearly $(X,h)$ is a minimal dynamical system and by Lemma \ref{ZsimpleZinvSubsets} we conclude that $C(X)$ is $\Z$-simple. However, each element of $X$ is $n$-periodic and hence $(X,h)$ is not topologically free, which by Lemma \ref{maxcommequivtopfree} entails that $C(X)$ is not maximal commutative in $\rring=C(X) \rtimes_{\tilde{h}} \Z$. The ring $\C[t,t^{-1}]$ is not simple (e.g. by Example \ref{LaurentGroupRing}) and via the isomorphism $\pi$ we conclude that $C(X) \rtimes_{\tilde{h}} \Z$ is never simple. From Section \ref{MiyashitaAction} it is clear that the action $\tilde{h}$ extends to an action of $\Z$ on $C_{\rring}(C(X))$. Finally, by Proposition \ref{GSimpleCommutantSimple}, we conclude that the commutant of $C(X)$ is never $\Z$-simple for our finite single orbit dynamical system.
\end{example}

%Formulera om minimalitet utan topologiska krav.

%\bibliographystyle{amsalpha}
%\begin{thebibliography}{A}
%
%\bibitem [A]{A} T. Aoki, \textit{Calcul exponentiel des op\'erateurs
%microdifferentiels d'ordre infini.} I, Ann. Inst. Fourier (Grenoble)
%\textbf{33} (1983), 227--250.
%
%\bibitem [B]{B} R. Brown, \textit{On a conjecture of Dirichlet},
%Amer. Math. Soc., Providence, RI, 1993.
%
%\bibitem [D]{D} R. A. DeVore, \textit{Approximation of functions},
%Proc. Sympos. Appl. Math., vol. 36,
%Amer. Math. Soc., Providence, RI, 1986, pp. 34--56.
%
%\end{thebibliography}

\bibliographystyle{amsalpha}

\begin{thebibliography}{99}

\bibitem{ArchSpiel} R. J. Archbold and J. S. Spielberg, \textit{Topologically free actions and ideals in discrete $C^*$-dynamical systems}, Proceedings of the Edinburgh Mathematical Society \textbf{37} (1993), 119--124.

\bibitem{CaenOyst} S. Caenepeel and F. Van Oystaeyen, {\it Brauer groups and the cohomology of graded rings}, Monographs and Textbooks in Pure and Applied Mathematics, 121, Marcel Dekker, Inc., New York, 1988.

\bibitem{Crow} K. Crow, \textit{Simple regular skew group rings}, J. Algebra Appl. \textbf{4} (2005), no. 2, 127--137.

\bibitem{Dade1982}
E. C. Dade, \textit{The equivalence of various generalizations of group rings and modules}, Math. Z. \textbf{181} (1982), no. 3, 335--344.

%\bibitem{CohenMontgomery} Cohen, M., Montgomery, S.: Group Graded Rings, Smash Products and Group Actions. Trans. Amer. Math. Soc. \textbf{282}, no. 1, 237--258 (1984)
%\bibitem{FisherMontgomery} Fisher, J.W., Montgomery, S.: Semiprime Skew Group Rings. J. Algebra \textbf{52}, no. 1, 241--247 (1978)
%\bibitem{FormanekLichtman} Formanek, E., Lichtman, A.I.: Ideals in Group Rings of Free Products. Israel J. Math. \textbf{31}, no. 1, 101--104 (1978)
%\bibitem{Irving1D} Irving, R.S.: Prime Ideals of Ore Extensions over Commutative Rings. J. Algebra \textbf{56}, 315--342 (1979)
%\bibitem{Irving2D} Irving, R.S.: Prime Ideals of Ore Extensions over Commutative Rings II. J. Algebra \textbf{58}, 399--423 (1979)
%\bibitem{HellSil} Hellstr\"om, L., Silvestrov, S.: Two-sided ideals in $q$-deformed Heisenberg algebras. Expo. Math. \textbf{23}, no. 2, 99--125 (2005)
%\bibitem{HellSilBook} Hellstr\"om, L., Silvestrov, S.D.: Commuting elements in $q$-deformed Heisenberg algebras. World Scientific Publishing Co., Inc., River Edge, New Jersey (2000)
\bibitem{Herstein} I. N. Herstein, {\it Noncommutative Rings}, xi+199 pp, The Carus Mathematical Monographs, no. 15, The Mathematical Association of America, New York, 1968.
%\bibitem{Jacobson} Jacobson, N.: Structure of Rings. American Mathematical Society Colloquium Publications, Vol. 37 (1964)
\bibitem{K} G. Karpilovsky, {\it The Algebraic Structure of Crossed Products}, x+348 pp, North-Holland Mathematics Studies, 142, Notas de Matem\'atica, 118, North-Holland, Amsterdam, 1987.
%\bibitem{Crow} Crow, Kathi...
%%\bibitem{Lang} Lang, S., \emph{Algebra}. Revised third edition. Corrected forth printing. Graduate Texts in Mathematics, 211. Springer-Verlag, New York, 2005.
%%\bibitem{temp} Lang, S., \emph{Algebra}. Revised third edition. Graduate Texts in Mathematics, 211. Springer-Verlag, New York, 2002.
%\bibitem{LaunLenaRiga} S. Launois, T. H. Lenagan and L. Rigal, \textit{Quantum unique factorisation domains}, J. London Math. Soc. \textbf{74}, no. 2, 321--340 (2006)
%\bibitem{LeroyMatczuk} Leroy, A., Matczuk, J.: Primitivity of Skew Polynomial and Skew Laurent Polynomial Rings. Comm. Algebra \textbf{24}, no. 7, 2271--2284 (1996)
%%\bibitem{LiBing-Ren} Li, B.-R., \emph{Introduction to Operator Algebras}. World Scientific Publishing Co., Singapore, 1992.
%\bibitem{LorenzPassman3} Lorenz, M., Passman, D.S.: Centers and Prime Ideals in Group Algebras of Polycyclic-by-Finite Groups. J. Algebra \textbf{57}, 355--386 (1978)
%\bibitem{LorenzPassman} Lorenz, M., Passman, D.S.: Prime Ideals in Crossed Products of Finite Groups. Israel J. Math. \textbf{33}, no. 2, 89--132 (1979)
%\bibitem{LorenzPassman2} Lorenz, M., Passman, D.S.: Addendum - Prime Ideals in Crossed Products of Finite Groups. Israel J. Math. \textbf{35}, no. 4, 311--322 (1980)
%\bibitem{MarubNauweOysta} Marubayashi, H., Nauwelaerts, E., Van Oystaeyen, F.: Graded Rings over Arithmetical Orders. Comm. Algebra \textbf{12}, no. 6, 774--775 (1984)
%\bibitem{McConnellRobson} McConnell, J.C., Robson, J.C.: Noncommutative Noetherian Rings. Pure \& Applied Mathematics, A Wiley-Interscience Series of Texts, Monographs, and Tracts, John Wiley \& Sons (1987)
%\bibitem{MontgomeryPassman} Montgomery, S, Passman, D.S.: Crossed Products over Prime Rings. Israel J. Math. \textbf{31}, nos. 3-4, 224--256 (1978)
\bibitem{MoGR} C. N\v ast\v asescu and F. Van Oystaeyen, {\it Methods of Graded Rings}, Lecture Notes in Mathematics, 1836, Springer-Verlag, Berlin, 2004.
%\bibitem{GenTwist} Nauwelaerts, E., Van Oystaeyen, F.: Generalized Twisted Group Rings. J. Algebra \textbf{294}, No. 2, 307--320 (2005)
\bibitem{CGR} E. Nauwelaerts and F. Van Oystaeyen, \textit{Introducing Crystalline Graded Algebras}, Algebr. Represent. Theory \textbf{11} (2008), No. 2, 133--148.
%\bibitem{NeOyYu} Neijens, T., Van Oystaeyen, F., Yu, W.W.: Centers of Certain Crystalline Graded Rings. Preprint in preparation (2007)
%\bibitem{OinSil} \"Oinert, J., Silvestrov, S.D.: Commutativity and Ideals in Algebraic Crossed Products. To appear in J. Gen. Lie. T. Appl. (2008)
%\bibitem{OinSilAGMFGBG} \"Oinert, J., Silvestrov, S.D.: On a Correspondence Between Ideals and Commutativity in Algebraic Crossed Products. Preprints in Mathematical Sciences 2007:30, ISSN 1403-9338, LUTFMA-5094-2007, Centre for Mathematical Sciences, Lund University (2007)

\bibitem{OinSil}
J. \"Oinert and S. D. Silvestrov, \textit{Commutativity and Ideals in Algebraic Crossed Products}, J. Gen. Lie. T. Appl. \textbf{2} (2008), no. 4, 287--302.

\bibitem{OinSilAGMFGBG}
\bysame
%\"Oinert, J., Silvestrov, S. D.
, \textit{On a Correspondence Between Ideals and Commutativity in Algebraic Crossed Products}, J. Gen. Lie. T. Appl. \textbf{2} (2008), No. 3, 216--220.

\bibitem{OinSil3}
\bysame
%\"Oinert, J., Silvestrov, S. D.
, \textit{Crossed Product-Like and Pre-Crystalline Graded Rings}, Chapter 24 in S. Silvestrov, E. Paal, V. Abramov, A. Stolin (Eds.), {\it Generalized Lie theory in Mathematics, Physics and Beyond}, pp. 281--296, Springer-Verlag, Berlin, Heidelberg, 2009.
%\"Oinert, J., Silvestrov, S., Crossed Product-Like and Pre-Crystalline Graded Rings, 16 pp. in Generalized Lie Theory in Mathematics, Physics and Beyond. Conference proceedings of Algebra, Geometry and Mathematical Physics, Baltic-Nordic Workshop (Lund, October 12-14 2006). Springer, (2008)

\bibitem{OinSil4}
\bysame
%\"Oinert, J., Silvestrov, S.
, \textit{Commutativity and Ideals in Pre-Crystalline Graded Rings}. To appear in Acta Appl. Math. (2009)
%\"Oinert, J., Silvestrov, S., Commutativity and Ideals in Pre-Crystalline Graded Rings. Preprints in Mathematical Sciences 2008:17, ISSN 1403-9338, LUTFMA-5101-2008, Centre for Mathematical Sciences, Lund University (2008)

\bibitem{OinSilTheVav}
J. \"Oinert, S. Silvestrov, T. Theohari-Apostolidi and H. Vavatsoulas, \textit{Commutativity and Ideals in Strongly Graded Rings}. To appear in Acta Appl. Math. (2009)
%\"Oinert, J., Silvestrov, S., Theohari-Apostolidi, T., Vavatsoulas, H., Commutativity and Ideals in Strongly Graded Rings. Preprints in Mathematical Sciences 2008:13, ISSN 1403-9338, LUTFMA-5100-2008, Centre for Mathematical Sciences, Lund University (2008)

%\bibitem{OinSilAGMFLund} \"Oinert, J., Silvestrov, S.D.: Crossed Product-like and Pre-Crystalline Graded Rings. To Appear in AGMF Network Conference Proceedings, Lund 2006, Springer (2008)
%%\bibitem{OinSil2} Oinert, J., Silvestrov, S.D.: Ideals in crossed products and skew group rings. In preparation.

\bibitem{OlesenPedersenII}
D. Olesen and G. K. Pedersen, \textit{Applications of the Connes spectrum to $C\sp{\ast} $-dynamical systems II.}, J. Funct. Anal. \textbf{36} (1980), no. 1, 18--32.

\bibitem{Power}
S. C. Power, \textit{Simplicity of $C\sp{\ast} $-algebras of minimal dynamical systems}, J. London Math. Soc. (2) \textbf{18} (1978), no. 3, 534--538.

%%\bibitem{Passman} Passman, D.S., \emph{Infinite Crossed Products}, Academic Press, 1989.
%\bibitem{Passman2} Passman, D.S.: The Algebraic Structure of Group Rings. Pure and Applied Mathematics. Wiley-Interscience (John Wiley \& Sons), New York-London-Sydney (1977)
\bibitem{Rowen} L. Rowen, \textit{Some Results on the Center of a Ring with Polynomial Identity}, Bull. Amer. Math. Soc. \textbf{79} (1973), no. 1, 219--223.

%\bibitem{Sierakowski} Sierakowski, A., The ideal structure of reduced crossed products, arXiv:0804.3772v1 [math.OA] (2008)

\bibitem{SSD1}
C. Svensson, S. Silvestrov and M. de Jeu, \textit{Dynamical Systems and Commutants in Crossed Products}, Internat. J. Math. \textbf{18} (2007), no. 4, 455--471.

\bibitem{SSD2}
\bysame, \textit{Connections Between Dynamical Systems and Crossed Products of Banach Algebras by $\mathbb{Z}$}, in {\it Methods of Spectral Analysis in Mathematical Physics, Conference on Operator Theory, Analysis and Mathematical Physics (OTAMP) 2006, Lund, Sweden}, Operator Theory: Advances and Applications, Vol. 186, Janas, J., Kurasov, P., Laptev, A., Naboko, S. and Stolz, G. (Eds.), pp. 391--401, Birkh\"auser, 2009.

\bibitem{SSD3}
\bysame, \textit{Dynamical systems associated to crossed products}, Preprints in Mathematical Sciences 2007:22, LUFTMA-5088-2007; Leiden Mathematical Institute report 2007-30; arxiv:0707.1881. To appear in Acta Appl. Math. (2009)

%\bibitem{TomiyamaNotes1} Tomiyama, J.: The interplay between topological dynamics and theory of $C\sp *$-algebras. Lecture Notes Series, 2. Global Anal. Research Center, Seoul (1992)
%\bibitem{TomiyamaNotes2} Tomiyama, J.: The interplay between topological dynamics and theory of $C^*$-algebras. Part 2 (after the Seoul lecture note 1992). Preprint (2000)
\bibitem{ChristianTomiyama} C. Svensson and J. Tomiyama, \textit{On the commutant of $C(X)$ in $C^*$-crossed products by $\Z$ and their representations}, J. Funct. Anal. {\bf 256} (2009), 2367--2386.
%arXiv:0807.2940v1 [math.OA] (2008)

\bibitem{ChristianTomiyama2}
\bysame,
\textit{On the Banach $*$-algebra crossed product associated with a topological dynamical system},
arXiv:0902.0690v1 [math.OA] (2009)

\bibitem{TomiyamaBook} J. Tomiyama, {\it Invitation to $C\sp *$-algebras and topological dynamics}, World Scientific Advanced Series in Dynamical Systems, 3. World Scientific Publishing Co., Singapore, 1987.

\bibitem{TomiyamaPaper} \bysame, \textit{$C\sp *$-algebras and topological dynamical systems}, Rev. Math. Phys. \textbf{8} (1996), 741--760.

\bibitem{VanOystaeyen} F. Van Oystaeyen, \textit{On Clifford Systems and Generalized Crossed Products}, J. Algebra \textbf{87} (1984), 396--415.
%\bibitem{Zeller-Meier} Zeller-Meier, G.: Produits crois\'es d'une $C\sp *$-alg\`ebre par un groupe d'automorphismes. J. Math. Pures Appl. \textbf{47}, 101--239 (1968)
\end{thebibliography}

\end{document}